\documentclass[11pt]{article}
\usepackage{amsmath,amssymb,amsthm}
\usepackage{graphicx}
\usepackage[in]{fullpage}
\usepackage{amsrefs}
\usepackage{todonotes}
\usepackage{mathtools}
\usepackage{hyperref}

\newcommand{\cA}{{\cal A}}
\newcommand{\calM}{{\cal M}}

\DeclareMathOperator{\diam}{diam}

\newtheorem{theorem}{Theorem} 
\newtheorem*{theorem*}{Theorem}

\newtheorem{claim}{Claim}

\theoremstyle{definition}

\title{On the Size of Finite Sidon Sets}
\author{Kevin O'Bryant \\ City University of New York, The Graduate Center and the College of Staten Island \\ \href{mailto:kevin.obryant@csi.cuny.edu}{kevin.obryant@csi.cuny.edu}}
\date{\today}

\begin{document}
\maketitle
\begin{abstract}
A Sidon set (also called a Golomb Ruler, a $B_2$ sequence, and a $1$-thin set) is a set of integers containing no nontrivial solutions to the equation $a+b=c+d$. We improve on the lower bound on the diameter of a Sidon set with $k$ elements: if $k$ is sufficiently large and $\cA$ is a Sidon set with $k$ elements, then $\diam(\cA)\ge k^2-1.99405 k^{3/2}$. Alternatively, if $n$ is sufficiently large, then the largest subset of $\{1,2,\dots,n\}$ that is a Sidon set has cardinality at most $n^{1/2}+0.99703 n^{1/4}$. While these are slight numerical improvements on Balogh-F\"uredi-Roy (arXiv:2103:15850v2), we use a method that is logically simpler than theirs but more involved computationally.
\end{abstract}

\section{Introduction}
A Sidon set~\cite{1932.Sidon} is a set $\cA$ of integers with no
solutions to $a_1+a_2=a_3+a_4,a_i\in \cA$, aside from the trivial
$\{a_1,a_2\}=\{a_3,a_4\}$. Equivalently, $\cA$ is a Sidon set if the
coefficients of $(\sum_{a\in \cA} z^a)^2$ are bounded by 2. It is also
equivalent, and more useful in this particular work, that $\cA$ is a Sidon set if there are no solutions to $a_1-a_2=a_3-a_4$ with $a_1 \neq a_2$ and $a_1 \neq a_3$. We refer our readers to~\cite{mybib} for a survey of Sidon sets and an extensive bibliography. Sidon sets are also known as $B_2$ sets, Golomb rulers, and 1-thin sets.

Throughout this work, $\cA=\{a_1<a_2<\dots<a_k\}$ is a finite Sidon set of integers. We seek lower bounds on the diameter $\diam(\cA):=a_k-a_1$. For a given $k$, the minimum possible value of $\diam(\cA)$ is denoted $s_k$. We let $R_2(n)$ be the maximum value of $k$ with $s_k\le n-1$, i.e., the maximum possible size of a Sidon set contained in $\{1,2,\dots,n\}$.

In 1938, Singer~\cite{1938.Singer} constructed a Sidon set with $q+1$ elements and diameter $q^2+q+1$, where $q$ is any prime power, whence $s_k \le k^2$ for infinitely many $k$. Alternatively, $R_2(n)> \sqrt{n}$ for infinitely many $n$. Using results on the distribution of primes (assuming RH), this construction gives $s_k \le k^2+k^{3/2+\epsilon}, R_2(n)\le \sqrt{n}+n^{1/4+\epsilon}$ for any $\epsilon>0$ and sufficiently large $k,n$. Alternative constructions that have the same asymptotic implications for $s_k,R_2(n)$ have been given by Bose~\cite{1942.Bose} and Ruzsa~\cite{Ruzsa}. Recently, Eberhard and Manners~\cite{Eberhard} have given a unified treatment that has these three constructions---and no others that are asymptotically as good---as special cases.

In 1941, Erd\H{o}s-Tur\'an~\cite{ET} use a ``windowing'' argument to derive the inequality (valid for any positive integers $n,T$)
\begin{equation}\label{eq:ETinequality} n-1 \ge \frac{R_2(n)^2 T}{T+R_2(n)-1}-T,\end{equation}
and from this inequality find that $R_2(n)=n^{1/2}+O(n^{1/4})$. In subsequent years, analysis of this inequality was improved and finally it was noted~\cite{Cilleruelo} that $R_2(n)=n^{1/2}+n^{1/4}+\frac12$ follows from~\eqref{eq:ETinequality}. In 1969, Lindstr\"om gave another proof, with a different inequality, which also gives this bound on $R_2(n)$. In~\ref{sec:ET} below, we give an even-more-elementary form of the Erd\H{o}s-Tur\'an argument and a ``book'' analysis of~\ref{eq:ETinequality} to give the sharpest expicit bound on $s_k$ that explicitly appears in the literature.
\begin{theorem}
    For $k\ge1$, $s_k \geq k^2 - 2k^{3/2}+k+\sqrt k-1.$
\end{theorem}

In light of the results cited above, the main term in $\diam(\cA)$ is $k^2$, and the secondary term is between $k^{3/2+\epsilon}$ and $-2k^{3/2}$. Recently, Balogh, F\"uredi, and Roy~\cite{BFR} (hereafter referred to as BFR) improved the bound on the secondary term from $-2k^{3/2}$ to $-1.996k^{3/2}$. This was the first real improvement on the bound on $\diam(\cA)$ or $R_2(n)$ since 1941. The number ($1.996$) appearing in their work is not fully optimized, and the important aspect of their work is that the number can be improved to below $2$ by combining the Erd\H{o}s-Tur\'an and Lindstr\"om proofs in a nontrivial manner. In this work we further improve the secondary term. The numerical improvement (also not fully optimized) is not the important aspect of this work, but that it is possible to improve the number below $2$ by a minor refinement the Erd\H{o}s-Tur\'an argument and not using the Lindstr\"om argument at all. Our method is logically simpler than the BFR method, but computationally more involved.
\begin{theorem}\label{thm:Main}
    Suppose that $\cA$ is a finite Sidon set of integers. If $k$ is sufficiently large, then $\diam(\cA) \ge k^2 - 1.99405 k^{3/2}$. Also, if $n$ is sufficiently large, then $R_2(n)<n^{1/2}+0.99703n^{1/4}$.
\end{theorem}
The ``sufficiently large'' phrase is not fundamentally needed. We expect further numerical improvements will appear in the next few years and it is premature to overly optimize at this stage.

We define $b_k$ by $s_k = k^2-b_k k^{3/2}$, and
$b_{\infty}=\limsup_{k\to\infty} b_k$. The Erd\H{o}s-Tur\'an work
proves that $b_k <2$, and the Balogh-F\"uredi-Roy work proves that $b_\infty\le 1.996$. Explicit values of $b_k$ (see Figure~\ref{fig:exact}) are known for $1\le k \le 27$ thanks to the Distributed.net ``Optimal Golomb Ruler'' project~\cite{Distnet}.

\begin{figure}
    \begin{center}
        \includegraphics[width=5in]{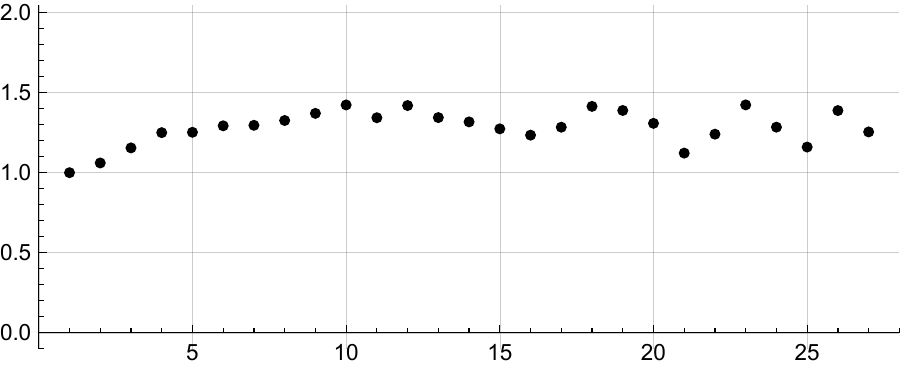}
    \end{center}
    \caption{Exact values of $b_k$ for $1\le k \le 27$.}
    \label{fig:exact}
\end{figure}

Shearer~\cite{Shearer} (perhaps around 1986) constructed the Sidon set examples of Singer and Bose for prime powers up to around 150, and thereby gave a reasonably sharp upper bound on $s_k$ for $k\le 150$. More recently, Dogon-Rokicki~\cite{DR} have taken this computation to its logical extreme. They computed the lower bound on $b_k$ that arises from looking at all subsets of modular dilations and translations of the sets described by Singer, Bose, and Ruzsa for all prime powers up to 40\,200. This reproduces the known optimal sets for all $k\ge 17$ (and most $k$ below 17, too). Their amazing data, which includes specific candidates for $k$-element Sidon sets with minimal diameter for $k<40\,000$, has been instrumental to this work. Their bound for $k<40\,000$ is shown in Figure~\ref{fig:DRbound}.
\begin{figure}
    \begin{center}
        \includegraphics[width=5in]{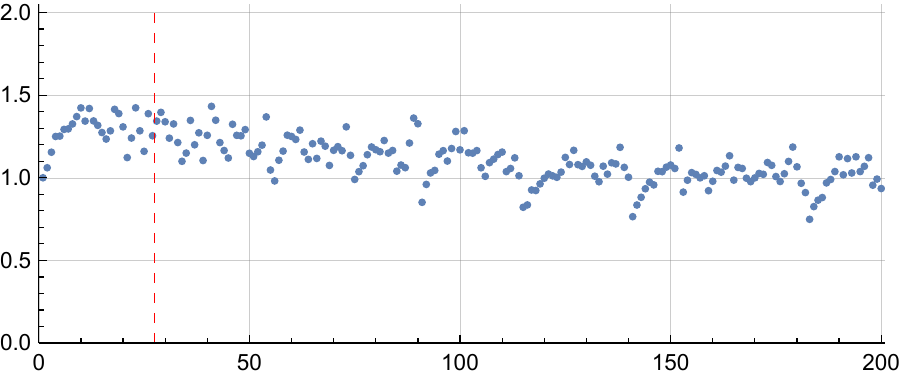}

        \includegraphics[width=5in]{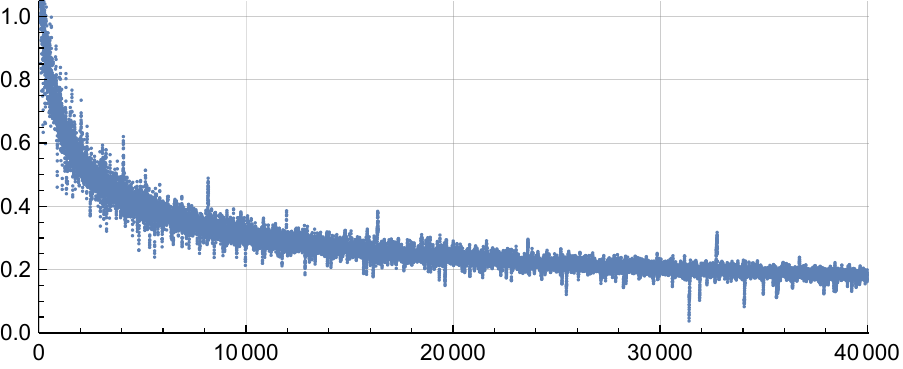}
    \end{center}
    \caption{Lower bound on $b_k$ for $1\le k \le 40\,000$, as computed by Dogon-Rokicki. The vertical red line separates the known exact values from the mere lower bounds.}
    \label{fig:DRbound}
\end{figure}

\section{The Erd\H{o}s-Tur\'an Sidon Set Equation}\label{sec:ET}
We first give a simplified account of the inequality that Erd\H{o}s-Tur\'an derived and used to prove $R(n)<n^{1/2}+O(n^{1/4})$, with a sharper derivation of the resulting bound. In fact, the bound given here is sharper than any explicit bound explicitly given in the literature of which this author is aware.

\begin{theorem}\label{thm:ErdosTuran2}
For $k\ge 1$, we have $s_k \geq k^2 - 2k^{3/2}+k+\sqrt k-1$. For $n\ge 1$, we have $R_2(n) < n^{1/2}+n^{1/4}+\frac12$.
\end{theorem}

\begin{proof}
    Let $\cA=\{0=a_1<a_2<\dots<a_k\}$ be a $k$-element Sidon set with minimum possible diameter. Let $A_j$ be the size of the intersection of $\cA$ with $[j-T,j)$.

    First, each of the $k$ elements of $\cA$ contributes to exactly $T$ of the $A_j$, so that
\begin{equation}\label{eq:first moment}
    \sum_{j=1}^{a_k+T} A_j = k T.
\end{equation}
By expanding $\big( A_j - \frac{kT}{a_k+T}\big)^2$ and using~\eqref{eq:first moment}, straightforward algebra now confirms that
\begin{equation}\label{eq:second moment}
    \sum_{j=1}^{a_k+T} A_j^2 = \frac{k^2T^2}{a_k+T} +
    \sum_{j=1}^{a_k+T} \bigg( A_j - \frac{kT}{a_k+T}\bigg)^2 \ge  \frac{k^2T^2}{a_k+T}.
\end{equation}

We can combine~\eqref{eq:first moment} and~\eqref{eq:second moment} to get
\begin{equation*}
    \sum_{j=1}^{a_k+T} \binom{A_j}{2} \ge \frac12 \frac{k^2T^2}{a_k+T} -\frac12 k T,
\end{equation*}
and this raises a new combinatorial interpretation. The binomial
coefficient $\binom{A_j}{2}$ is counting the number of pairs of
elements of $\cA$ in the interval $[j-T,j)$.
Each pair $(a_y,a_x)$ of elements of $\cA$ with $a_x>a_y$ contributes 1 to $\binom{A_j}{2}$ for $j \in [a_x+1,a_y+T]$ if $r=a_x-a_y \le T-1$, and contributes nothing to any $\binom{A_j}{2}$ if $r\ge T$. By the Sidon property, if $r\in \cA-\cA$, then there is a unique pair $(a_y,a_x)$ with $r=a_x-a_y$. Ergo, as $|[a_x+1,a_y+T]|=T-r$, we have
\begin{equation}\label{eq:Sidon property used}
    2\sum_{j=1}^{a_k+T} \binom{A_j}{2} = \sum_{\substack{r=1 \\ r \in \cA-\cA}}^{T-1} T-r
    = \frac{T(T-1)}{2} - \sum_{\substack{r=1 \\ r \not\in \cA-\cA}}^{T-1} T-r \le \frac{T(T-1)}2.
\end{equation}

Comparing the upper and lower bounds on $\sum \binom{A_j}{2}$, we now have
    \begin{equation*}
    \frac12 \frac{k^2T^2}{a_k+T} - \frac{kT}{2} \ge \frac{T(T-1)}{2}.
    \end{equation*}
Solving for $a_k$ gives the Erd\H{o}s-Tur\'an Sidon Set Inequality: For positive integers $T$,
    \[a_k \ge \frac{k^2 T}{T+k-1}-T.\]

Set $T=k^{3/2}-k+\epsilon$, with $\epsilon\in(0,1]$ chosen to make $T$ a positive integer.
We have the factorization
\begin{equation*}
      \left(\frac{k^2 T}{T+k-1}-T\right) - \big( k^2 - 2k^{3/2}+k+\sqrt k - 1 \big)
       = \frac{(1-\epsilon ) \big(\sqrt{k}+\epsilon -1\big)}{T+k-1},
\end{equation*}
which is clearly nonnegative for $k\ge 1$. Therefore, $s_k=a_k \ge k^2 - 2k^{3/2}+k+\sqrt k-1$, as claimed.
\end{proof}

We can convert this to a bound on $k=R_2(n)$ easily. The expression $k^2 - 2k^{3/2}+k+\sqrt k-1$ is a polynomial in $\sqrt k$, and one easily sees that its derivative (with respect to $\sqrt k$) is positive for $k\ge0$. Thus, the inequality $n-1\ge k^2 - 2k^{3/2}+k+\sqrt k-1$ implies an upper bound on $k$. To prove that $R_2(n)<n^{1/2}+n^{1/4}+\frac12$, for example, one only needs to derive a contradiction from
\[ n-1 \ge (n^{1/2}+n^{1/4}+\frac12)^2 - 2 (n^{1/2}+n^{1/4}+\tfrac12)^{3/2}+(n^{1/2}+n^{1/4}+\tfrac12)+\sqrt{n^{1/2}+n^{1/4}+\tfrac12}-1.\]
While tedious, this purported inequality simplifies algebraically and can be proved impossible by hand.

This proof contains two innovations and a useful consequence. The first innovation is on~\eqref{eq:second moment}, where we used algebra instead of Cauchy's Inequality. The algebra of course constitutes a proof of Cauchy's inequality, but the crux of this article rests on this nontrivial change. The second innovation is the discovery that this bound on $s_k$ produces a pleasantly useful factorization, allowing us to also replace asymptotic analysis entirely with simple algebra.

The consequence is made possible by the first innovation. We can use the equalities in~\eqref{eq:second moment} and~\eqref{eq:Sidon property used} instead of the inequalities, and we arrive at the following theorem, which drives the rest of this work.
\begin{theorem}\label{thm:ETexact}
Let $\cA$ be a finite Sidon set of integers, and $T$ be a positive integer. Then
\begin{equation}
\diam(\cA) = \frac{|\cA|^2 T^2}{T(T+|\cA|-1)-(2S(\cA,T)+V(\cA,T))}-T,
\end{equation}
where
    \begin{align*}
    S(\cA,T) &:= \sum_{\substack{r=1\\r\not\in \cA-\cA}}^{T-1} (T-r),\\
    V(\cA,T) &:=\sum_{i=\min\cA+1}^{T+\max\cA} \left( A_i - \frac{kT}{T+\max \cA}\right)^2,
    \end{align*}
where $A_i := | \cA \cap [i-T,i)|$.
\end{theorem}

We find it philosophically interesting that Theorem~\ref{thm:ETexact} gives an \emph{equality}. Given $|\cA|$, any lower bound on $\diam(\cA)$ is equivalent to a lower bound on $2S(\cA,T)+V(\cA,T)$, and vice versa. Know any three of $|\cA|,\diam(\cA),V(\cA,T),S(\cA,T)$, and the fourth is uniquely determined.

For $i\le T$, one has $A_i \le R_2(i)$, and this gives a ``trivial'' lower bound on $V(\cA,T)$. This allows Theorem~\ref{thm:ETexact} to improve the inequality $R_2(n)<n^{1/2}+n^{1/4}+\frac12$ explicitly by $O(1)$. We mention this only because the number $1/2$ has been called the limit of what can be accomplished with the Erd\H{o}s-Tur\'an style argument. 

\subsection{Exploring the Erd\H{o}s-Tur\'an Sidon Set Equation}
This subsection contains no logic, but explores the dataset of Dogon-Rokicki~\cite{DR} to motivate the rest of this work, and perhaps also suggest avenues for further progress.

Asymptotic analysis gives that if $2S(\cA,T)+V(\cA,T) \ge \nu k T$ (where $\nu\in (0,1)$ is fixed and $T=\tau k^{3/2}$ is allowed to vary), then $$\diam(\cA) \ge k^2 - 2\sqrt{1-\nu} k^{3/2}+O(k).$$ In particular, we need to find a bound on $2S+V$ that is $\Omega(kT)$ to impact the $k^{3/2}$ coefficient. In Figure~\ref{fig:nupic}, we show $V(\cA,k^{3/2})/k^{5/2}$ for $k<4000$.
\begin{figure}
    \begin{center}
        \includegraphics[width=5in]{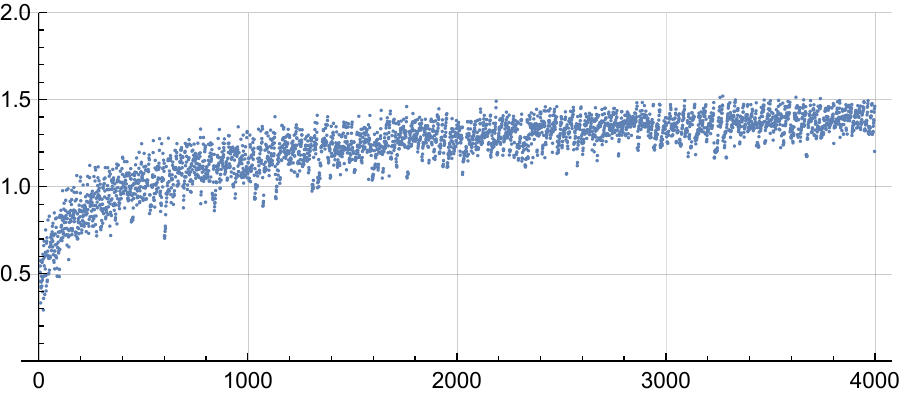}
    \end{center}
    \caption{$V(\cA,k^{3/2})/ k^{5/2}$ for the Sidon sets $\cA$ in the Dogon-Rokicki dataset ($28\le k <4000$).}
    \label{fig:nupic}
\end{figure}

In contrast to $V(\cA,T)$, which seems to be consistently more than $kT$ in the Dogon-Rokicki dataset, the contribution from $S(\cA,T)$ is small and irregular. See Figure~\ref{fig:nupic2}. The nearly vertical patterns in the data points in Figure~\ref{fig:nupic} and~\ref{fig:nupic2} are not visual artifacts, but possibly arise from many of the sets with close $k$ having small symmetric difference, and so similar values of $V,S$.
\begin{figure}[h]
    \begin{center}
        \includegraphics[width=5in]{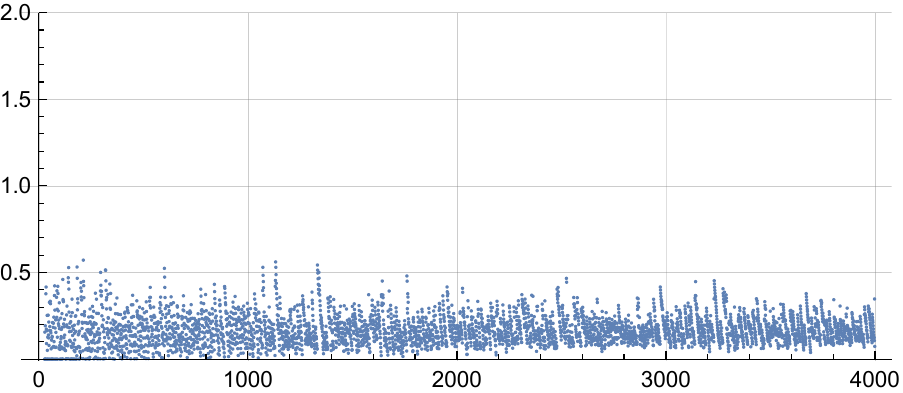}
    \end{center}
    \caption{$2S(\cA,k^{3/2})/ k^{5/2}$ for the Sidon sets $\cA$ in the Dogon-Rokicki dataset ($28\le k <4000$).}
    \label{fig:nupic2}
\end{figure}

There is a negative correlation between $V(\cA,T)$ and $S(\cA,T)$, as can be seen in Figures~\ref{fig:nupic3} and~\ref{fig:correlation}.
\begin{figure}[h!]
    \begin{center}
        \includegraphics[width=5in]{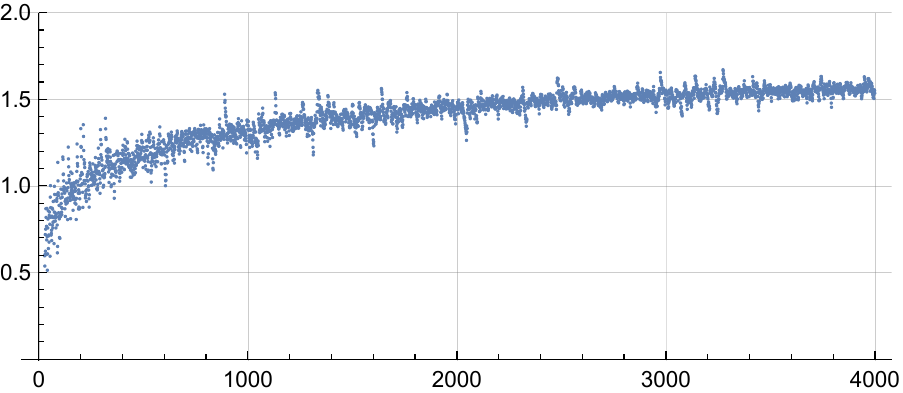}
    \end{center}
    \caption{$(V(\cA,k^{3/2})+2S(\cA,k^{3/2})/ k^{5/2}$ for the Sidon sets $\cA$ in the Dogon-Rokicki dataset ($28\le k <4000$). Notice that this graph appears smoother than those in Figures~\ref{fig:nupic} and~\ref{fig:nupic2}, showing the negative correlation.}
    \label{fig:nupic3}
\end{figure}[h!]
\begin{figure}[t!]
    \begin{center}
        \includegraphics[width=5in]{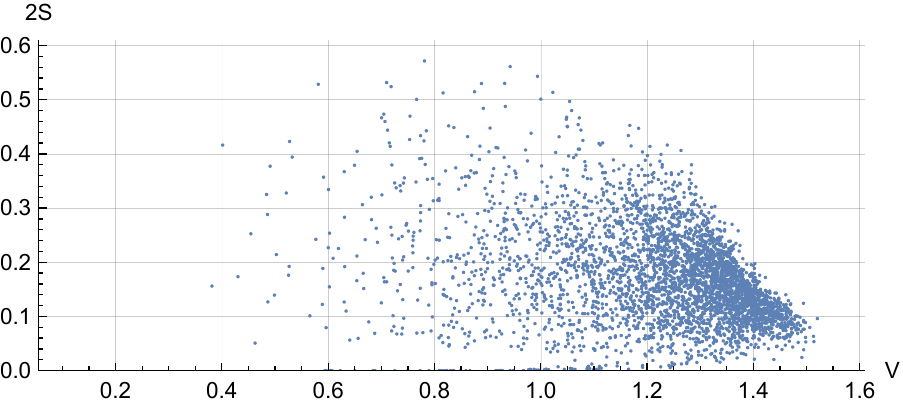}
    \end{center}
    \caption{The points $(V(\cA,k^{3/2}),2S(\cA,k^{3/2})/ k^{5/2}$ for the Sidon sets $\cA$ in the Dogon-Rokicki dataset ($28\le k <4000$). Notice that when $V$ is particularly large, $2S$ is particularly small.}
    \label{fig:correlation}
\end{figure}

From these images, it seems apparent that one should focus more on $V$ than $S$. In Figure~\ref{fig:A detail} we show four plots of $A_i:=|\cA \cap [i-k^{3/2},i)|$, with the average value as a red dashed horizontal line, and $i=T$,$i=a_k$ as vertical black lines. The four, which appear to be typical, show that $A_i$ is fairly smooth between $T$ and $a_k$. In fact, about half of $V(\cA,T)$ arises from $i\le T$ and $i\ge a_k$. Figure~\ref{fig:edgevariance} shows the proportion of $V$ that arises from the edge values of $i$.
\begin{figure}
    \begin{center}
        \includegraphics[width=5in]{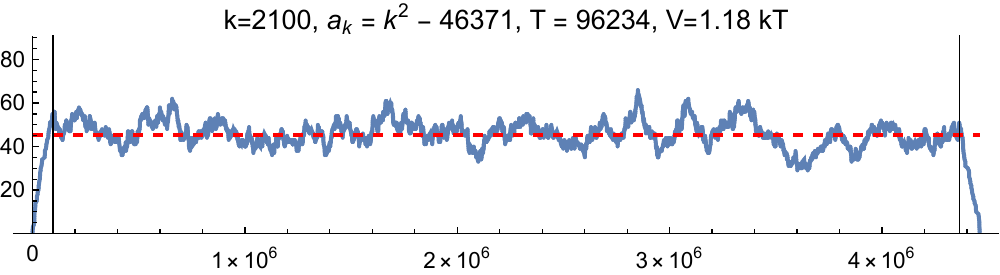}

        \includegraphics[width=5in]{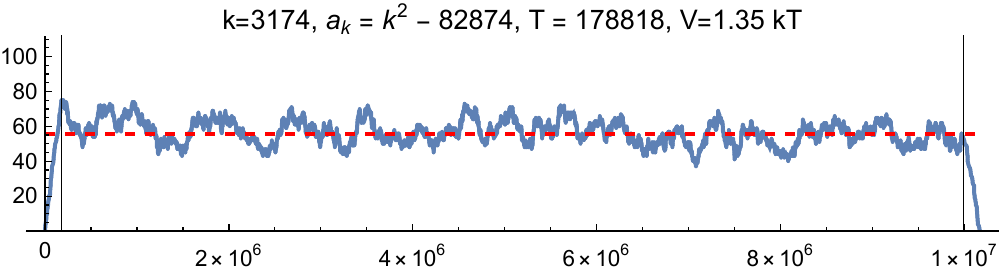}

        \includegraphics[width=5in]{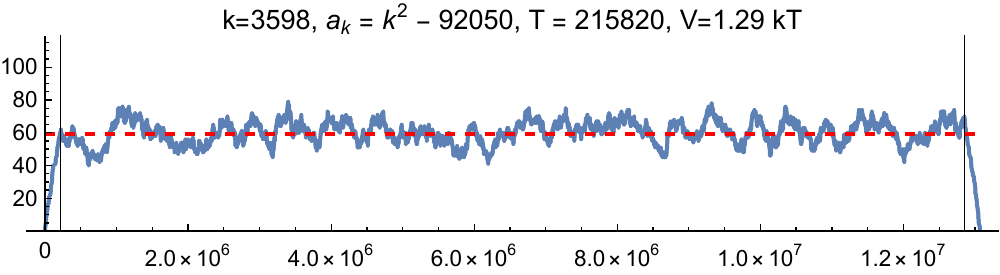}

        \includegraphics[width=5in]{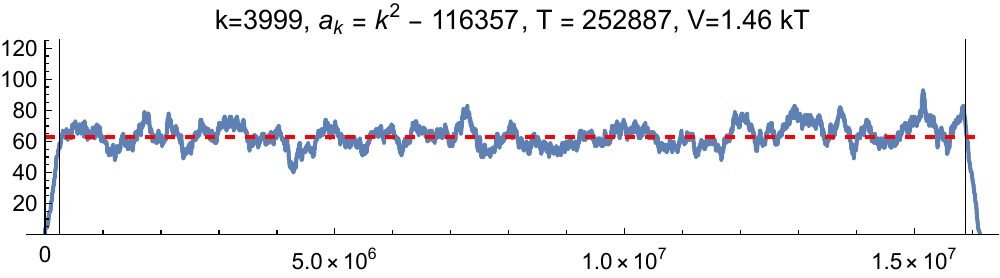}
    \end{center}
    \caption{Plot of $A_i$ as a function of $i$ for four typical thick Sidon sets. The average value of $A_i$ is shown as a horizontal red line, and the values $i=T$ and $i=a_k$ are shown as vertical black lines.}
    \label{fig:A detail}
\end{figure}
\begin{figure}
    \begin{center}
        \includegraphics[width=5in]{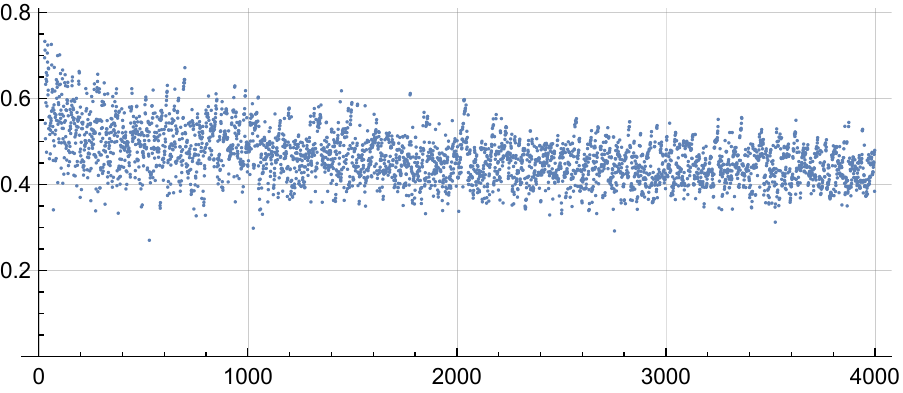}
    \end{center}
    \caption{$V(\cA,k^{3/2})$ is defined by a sum with $i$ running from 1 to $a_k+k^{3/2}$. This graph shows the proportion of $V(\cA,k^{3/2})$ that comes from $i\le T$ and $i\ge a_k$.}
    \label{fig:edgevariance}
\end{figure}

\section{The BFR Gambit}
In~\cite{BFR}, a small variance (which arose by way of an inequality from coding theory, and a rephrasing of the Erd\H{o}s-Tur\'an argument in terms of degree sequences of hypergraphs) was used to imply a lumpiness in the difference set. This lumpiness was then used to supercharge a proof of Lindstr\"om's, yielding
\[ R_2(n) < n^{1/2}+0.998 n^{1/4} \text{ for sufficiently large $n$.}\]
It is, as the proven value was rounded up to $0.998$, equivalent to
\[ s_k \ge k^2 - 1.996 k^{3/2} \text{ for sufficiently large $k$.}\]
This work was the first real improvement in the bounds, upper or lower, of $R_2(n)$ since 1941. It was the direct inspiration for the current work.

We play the same gambit, but exploit the phenomenon in a different way. Instead of using the Lindstr\"om proof, we re-use Theorem~\ref{thm:ETexact} on a subset of $\cA$. That is, either $\cA$ has a large variance (and so a large diameter), or a large subset $\calM \subseteq \cA$ has a difference set that does not include many small values, and so the small diffs is large, forcing $\calM$ to have a large diameter.

\subsection{A Toy Version of our Argument}
In this short subsection, we give a simple argument using the Erd\H{o}s-Tur\'an Sidon Set Equation to prove that $b_\infty \le 1.999$. In the next subsection, we engage in a somewhat harder-to-optimize version of this argument with more parameters that leads to a better bound on $b_\infty$ than that given in BFR.

Normalize as above, so that $a_1=0$, and set $T=k^{3/2}$ (it needs to be an integer, but we will ignore this issue in this subsection, and we shall also be cavalier about inequalities that may only hold for sufficiently large $k$). For positive $i$, set $A_i:=|\cA \cap [i-T,i)|$ and $A_{-i}:=|\cA \cap [a_k-i+1,a_k+T+1-i)|$, and let $\bar A := kT/(T+a_k)$, the average value of $A_i$ for $1\le i \le a_k+T$. If $a_k\le k^2-T$, then $\bar A \sim \sqrt k$. We begin by considering $(A_i+A_{-i})/2$ as a function of $i$; it is monotone increasing for $1\le i \le T$. Set
\begin{align*}
    U_1 &:= \left\{ 1\le j \le T : \tfrac{A_i+A_{-i}}2 \le \tfrac89 \bar A \right\} \\
    U_2 &:=  \left\{ 1\le j \le T : \tfrac89 \bar A < \tfrac{A_i+A_{-i}}2 < \tfrac{10}9 \bar A \right\} \\
    U_3 &:=  \left\{ 1\le j \le T : \tfrac{10}9 \bar A \le \tfrac{A_i+A_{-i}}2 \right\}
\end{align*}

If $|U_1|+|U_3| \ge T/16$, then
\begin{align*}
    V(\cA,T) &:= \sum_{i=1}^{T+a_k} (A_i - \bar A)^2 \\
    &\ge \sum_{i\in U_1\cup U_3} \left[  (A_i - \bar A)^2  + (A_{-i}-\bar A)^2 \right] \\
    &\ge \sum_{i\in U_1\cup U_3}  2 (\tfrac{A_i+A_{-i}}2 - \bar A)^2 \\
    &\ge \left(|U_1\cup U_2|\right) \cdot 2\left(\tfrac19\right)^2\bar A^2 \geq \frac{1}{648} kT.
\end{align*}

We can now appeal to Theorem~\ref{thm:ETexact}:
\[\diam(\cA) \ge \frac{k^2 T^2}{T(T+k-1)-\frac{1}{16}\frac{2}{81} kT}-T = k^2 - (2-\frac{1}{648})k^{3/2}+O(k).\]
Thus, $b_\infty \le 2-\frac{1}{648}\le 1.999$ if $|U_1|+|U_3|\ge \frac T{16}$.

Now suppose that $|U_1|+|U_3|\le \frac T{16}$, so that $|U_2|\ge \frac{15}{16}T$. Set
$${\calM}:= \cA \cap [|U_1|+|U_2| , a_k-|U_1|-|U_2|].$$
Clearly
\[|\calM| = k - A_{|U_1|+|U_2|} - A_{-(|U_1|+|U_2|)} \ge k-\frac{20}9 \bar A.\]

Set $L=\cA \cap[0,|U_1|)$ and $R=\cA \cap (a_k-|U_1|,a_k]$. By the definition of $U_1$, we see that
\[|L \cup R| \ge 2 \cdot \frac 89 \bar A.\]
We also see that $L\cup R$ has $\binom{|L|}{2}+\binom{|R|}{2}$ pairs of elements whose difference is at most $|U_1|$. We have
$$\binom{|L|}{2}+\binom{|R|}{2} \ge (\frac89 \bar A)^2-(\frac89 \bar A) \sim \frac{64}{81}k$$ differences, each at most $T/16$. Set $T_2=\frac38 T$, which is larger than $T/16$, so that
\begin{align*}
    S(\calM,T_2) &:= \sum_{\substack{r=1 \\ r\not\in\calM-\calM}}^{T_2-1} (T_2-r) \\
    &\ge \sum_{\substack{r=1 \\ r\in(L-L)\cup(R-R)}}^{T_2-1} (T_2-r) \\
    &\ge \frac{64}{81}k \cdot \left(\frac38 T - |U_1|\right) 
\end{align*}

Theorem~\ref{thm:ETexact} now gives
\begin{align*}
    \diam(\cA)&= \left(\min\calM - a_1\right) + \diam(\calM) + \left(a_k-\max\calM \right) \\
    &\ge 2(|U_1|+|U_2|) + \frac{|\calM|^2 T_2^2}{T_2(T_2+|\calM|-1)-2\cdot \left( \frac{64}{81}k \cdot \left(\frac38 T - |U_1|\right) \right)}-T_2 \\
    &= k^2 - \left(\frac{6734}{729}\frac{|U_1|}{T}-2 \frac{|U_2|}{T}+\frac{6361}{1944} \right) k^{3/2}
\end{align*}      
The extremal values for $|U_1|,|U_2|$ are $\frac{1}{16}T$ and $\frac{15}{16}T$, respectively, whence
$$b_\infty \le \frac{11515}{5832} < 1.999.$$

To get better numbers, we should have more levels than just $8/9,10/9$, and we should allow $T=\tau k^{3/2}$ to vary. We pursue this in the next subsection.

\subsection{More Parameters Leads to a Better Bound}

We begin by fixing real numbers $\tau>0$, $\alpha\in (0,1)$, $\beta\in (0,1)$, $\delta\in (0,1/2)$, $\tau_2>0$. Specific values that work out well are:
\begin{align}
    \label{set tau}   \tau&=\frac{59}{58}  \approx 1.017\\
    \label{set alpha} \alpha&=\frac{80}{319}  \approx 0.2508\\
    \label{set beta}  \beta&=\frac{195}{356}  \approx 0.5478\\
    \label{set delta} \delta&= \frac{398773753333438270}{2448810518987915261} \approx 0.1628\\
    \label{set tau2}  \tau_2&=\frac{51}{223} \approx 0.2287
\end{align}
The reason for the ridiculous value of $\delta$ is to make two bounds we will encounter exactly equal.

To better wrangle the equations, set $\min\cA=0$, $|\cA|=k$, $T=\tau k^{3/2}+\epsilon_1$ (all of the $\epsilon$ variables are chosen in $[0,1)$ to make some other variable an integer), and $\bar A=\frac{k T}{a_k+T}$, and assume that $s_k=a_k \le k^2-T$, so that $\tau \sqrt k \le \bar A < \tau \sqrt k + \epsilon_1/k$.

For $1\le j \le a_k+T$, we have $A_j := | \cA \cap [j-T,j) |$. For negative subscripts, we set $A_{-j}:= A_{a_k+T+1-j}$. For example, $A_1=|\{0\}| = 1 $ and $A_{-1}=|\{a_k\}|=1$. As a function of $j$, we see that $(A_j+A_{-j})/2$ is nondecreasing for $1\le j \le T$, and increases by at most 1 as $j$ increases by 1.

\begin{figure}[t]
    \begin{center}
        \includegraphics[width=5in]{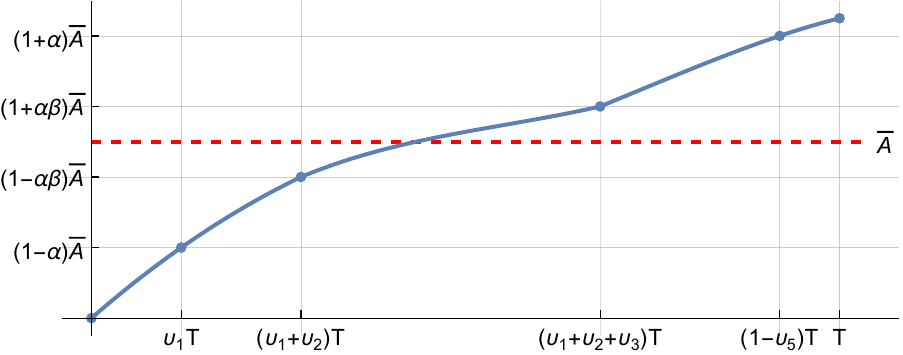}
    \end{center}
    \caption{Imagined plot of $(A_j+A_{-j})/2$ as a function of $j$, with the meanings of $\upsilon_1, \upsilon_2, \upsilon_3, \upsilon_4, \upsilon_5$ depicted.}
\end{figure}

We partition the interval $[1,T]$ into 5 parts:\footnote{Setting $\beta=1$ forces $U_2=U_4=\emptyset$, and is thereby substantially simpler, but leads merely $b_\infty \le 1.999$.}
\begin{align*}
    U_1 &:= \left\{ j \in [1,T]: \tfrac{A_j+A_{-j}}{2} \le (1-\alpha)\bar A \right\},\\
    U_2 &:= \left\{ j \in [1,T]: (1-\alpha)\bar A < \tfrac{A_j+A_{-j}}{2} \le (1-\alpha\beta)\bar A \right\},\\
    U_3 &:= \left\{ j \in [1,T]: (1-\alpha\beta)\bar A < \tfrac{A_j+A_{-j}}{2} \le (1+\alpha\beta)\bar A \right\},\\
    U_4 &:= \left\{ j \in [1,T]: (1+\alpha\beta)\bar A \le \tfrac{A_j+A_{-j}}{2} \le (1+\alpha)\bar A \right\},\\
    U_5 &:= \left\{ j \in [1,T]: (1+\alpha)\bar A < \tfrac{A_j+A_{-j}}{2} \right\},
\end{align*}
We set $u_i$ by $u_i T = |U_i|$, and set $y:=u_1+u_5, x:=u_2+u_4$. For example, $x+y+u_3=1$ and $u_1+u_2+u_3 \ge 1-x-y$.

\begin{claim}
    If $y+\beta^2x\ge \beta^2\delta$, then $V(\cA,T)\ge 2\alpha^2\beta^2\tau^2 \, kT$.
\end{claim}

\begin{proof}[Proof of Claim:]
\begin{align*}
    V(\cA,T)
    &\ge \sum_{j=1}^T (A_j-\bar A)^2 + (A_{-j}-\bar A)^2 \\
    &\ge \sum_{j=1}^T 2\left(\tfrac{A_j+A_{-j}}{2} - \bar A\right)^2 \\
    &= 2 \sum_{i=1}^5 \sum_{j\in U_i} \left(\tfrac{A_j+A_{-j}}{2} - \bar A\right)^2\\
    &\ge 2(|U_1|+|U_5|) \alpha^2 \bar A^2+2(|U_2|+|U_4|)(\alpha\beta)^2 \bar A^2 \\
    &\ge 2\alpha^2\tau^2(y+\beta^2x) \, kT.
\end{align*}
By hypothesis of this claim, $y+\beta^2x \ge \beta^2 \delta$, and so $V(\cA,T)\ge 2\alpha^2\beta^2\tau^2\delta \,kT$.
\end{proof}

\begin{claim}\label{claim:bound from variance}
    If $y+\beta^2x\ge \beta^2\delta$, then
    $$b_\infty \le \tau +\frac{1}{\tau } -2 \alpha ^2\beta^2 \delta \tau  \le \frac{3869247756486775922024264545}{1940405707787319054606925942} \le 1.99405.$$
\end{claim}

\begin{proof}[Proof of Claim:]
Theorem~\ref{thm:ETexact} gives (under the hypothesis that $y+\beta^2 x \ge \beta^2\delta$)
\[\diam(\cA)\ge \frac{k^2 T^2}{T(T+k-1)-2\alpha^2\tau^2\delta kT}-T
    = k^2 -\left(\tau +\frac{1}{\tau } -2 \alpha ^2\beta^2 \delta \tau\right)k^{3/2} + O(k).\]
With the values of $\tau,\alpha,\delta$ given on lines~\eqref{set tau},~\eqref{set alpha}, and~\eqref{set delta} above, this gives the claimed number.
\end{proof}

We now work under the hypothesis that $y+\beta^2 x \le \beta^2\delta$. Note that $|U_1|+|U_2|\le yT+xT \leq \left(\frac{y}{\beta^2}+x\right)\le \delta T<T$, so that necessarily $U_3$ is not empty. It can happen that $U_4,U_5$ are empty.

Let $X-X$ be the set of positive differences of elements of $X$. In particular, if $X$ is a subset of a Sidon set, then $|X-X|=\binom{|X|}2$. We set
$$\calM := \cA \cap(\max U_3,a_k-\max U_3)$$
We will find that $|\cal M|\approx|\cA|$, but that if $x,y$ are small then $\calM-\calM$ is missing a substantial number of small integers, so that $S(\calM,T_2)$ is $\Omega(k^{5/2})$, which forces the diameter of $\calM$ to be large. This, in turn, forces $\diam(\cA)\ge 2\max U_3+\diam(\calM)$ to be large, and this gives a bound on $b_\infty$.

\begin{claim}\label{claim:M}
    Suppose that $y+\beta^2x\le \beta^2\delta$.
Then $|\calM| \ge k-2(1+\alpha\beta)\tau\sqrt k + O(1)$.
\end{claim}

\begin{proof}[Proof of Claim:]
    As $y+\beta^2 x \le \beta^2\delta$, we know that $U_3$ is nonempty. Let $z=\max U_3$. By the definition of $U_3$, we know that $\frac{A_{z}+A_{-z}}{2}\leq (1+\alpha\beta)\bar A$. Therefore,
    $$|\calM| \ge k - 2(1+\alpha\beta)\bar A -2=k-2(1+\alpha\beta)\tau\sqrt k +O(1).$$
\end{proof}

We set $L_i:= \cA\cap(U_i-1)$ and $R_i:= \cA \cap (a_k+T+1-U_i)$. Note that $L_1,L_2,R_1,R_2$ are disjoint, and disjoint from $\calM$. But they are all subsets of the Sidon set $\cA$, so any difference in $(L_1\cup L_2)-(L_1\cup L_2)$ or $(R_1\cup R_2)-(R_1\cup R_2)$ is necessarily missing from $\calM-\calM$, and so contributes to $S(\calM,T_2)$ (provided $T_2$ is large enough).

\begin{claim}\label{claim:differences}
    Suppose that $y+\beta^2x\le \beta^2\delta$.

    There are at least $(1-\alpha)^2\tau^2 k+O(\sqrt k)$ elements of $(L_1-L_1) \cup (R_1-R_1)$, each at most $yT$.

    There are at least $\alpha^2(1-\beta)^2\tau^2 k+O(\sqrt k)$ elements of $(L_2-L_2) \cup (R_2-R_2)$, each at most $xT$.

    There are at least $(1-\alpha\beta)^2\tau^2k+O(\sqrt k)$ elements of $\big((L_1\cup L_2)-(L_1\cup L_2)\big) \cup \big((R_1\cup R_2)-(R_1\cup R_2)\big)$, each at most $(x+y)T$.
\end{claim}

\begin{proof}[Proof of Claim:]
We note that
$|L_1| =A_{\max U_1}, |R_1|=A_{-\max U_1},$ and the definition of $U_1$ (together with the nonemptiness of $U_2$, which is implied by $\frac{y}{\beta^2}+x \le \delta$) gives
$$A_{\max U_1}+A_{-\max U_1}\leq 2(1-\alpha)\bar A< A_{1+\max U_1}+A_{-1-\max U_1}\le A_{\max U_1}+A_{-\max U_1}+2,$$
and so
$$ 2(1-\alpha)\bar A-2 \le |L_1 \cup R_1| \le 2(1-\alpha)\bar A.$$
Similarly,
$$|L_2\cup R_2| = 2\alpha(1-\beta)\bar A +O(1),$$
$$|L_1\cup L_2 \cup R_2 \cup R_1| = 2(1-\alpha \beta)\bar A+O(1).$$

For a Sidon set $X$, the size of $X-X$ is exactly $\binom{|X|}{2}$. For example, $(L_1-L_1)\cup(R_1-R_1)$ has
$$\binom{|L_1|}2+\binom{|R_1|}2 \ge 2\binom{(1-\alpha)\bar A}2 = ((1-\alpha)\bar A)^2+O(\sqrt k)=(1-\alpha)^2\tau^2 k+O(\sqrt k)$$
elements, each at most $\diam( U_1)=u_1 T$, which is at most $y T$.

Essentially identical computations give the other assertions of this claim.
\end{proof}

We set $T_2:=\tau_2T+\epsilon_2$, with $\epsilon_2\in[0,1)$ chosen to make $T_2$ an integer.

\begin{claim}\label{claim:S}
    Suppose that $y+\beta^2x\le \beta^2\delta$ and $\tau_2>x+y$. Then $S(\calM,T_2)$ is at least    $kT\cdot \tau^2\cdot w$, where
    \begin{align}\label{eq:w}
        w&:=\tau_2(1-\alpha\beta)^2 - y (1-\alpha ) (1 +\alpha-2 \alpha  \beta) - x \alpha  (1-\beta ) (2-\alpha-\alpha\beta)\\
        &= \frac{30590263131}{179748900463}-\frac{6622929}{9056729}y-\frac{5081160 }{27794789}x
    \end{align}
\end{claim}

\begin{proof}[Proof of Claim:]
By hypothesis, $T_2=\tau_2 T > (x+y)T$, so all of the differences described in Claim~\ref{claim:differences} contribute to
\begin{align*}
    S(\calM,T_2) & := \sum_{\substack{r=1\\r\in \calM-\calM}}^{T_2-1}(T_2-r) \\
    &\ge \sum_{r\in \big((L_1\cup L_2)-(L_1\cup L_2)\big) \cup \big((R_1\cup R_2)-(R_1\cup R_2)\big)} (T_2-r).
\end{align*}

For example, the $(1-\alpha)^2\tau^2 k$ differences in $(L_1-L_1)\cup(R_1-R_1)$ contribute at least
\begin{equation}\label{eq:U_1 contributes}
    (T_2-yT)\cdot \big((1-\alpha)^2\tau^2 k+O(\sqrt k)\big) =(\tau_2-y)(1-\alpha)^2\tau^2 kT + O(k^2)
\end{equation}
to $S(\calM,T_2)$. Similarly, the $$\alpha^2(1-\beta)^2\bar A^2+O(\sqrt k)$$ differences in $(L_2-L_2)\cup(R_2-R_2)$, each at most $\diam(U_2) =u_2T < xT$, contribute at least
\begin{equation}\label{eq:U_2 contributes}
    (T_2-\diam(U_2)) \cdot \big(\alpha^2(1-\beta)^2\bar A^2+O(\sqrt k)\big)
    =(\tau_2 - x)\alpha^2(1-\beta)^2\tau^2 kT + O(k^2)
\end{equation}
to $S(\calM,T_2)$. Finally, $(L_1\cup L_2)-(L_1\cup L_2)$ and $(R_1\cup R_2)-(R_1\cup R_2)$ have many differences that haven't been counted already. They have
$$((1-\alpha\beta)\bar A)^2+O(\sqrt k)=(1-\alpha\beta)^2\tau^2k+O(\sqrt k)$$
elements, but only
\begin{equation*}
    (1-\alpha\beta)^2\tau^2 k - (1-\alpha)^2\tau^2 k-\alpha^2(1-\beta)^2
    \tau^2 k+O(\sqrt k) \\
    = 2 \alpha (1-\alpha)(1-\beta)\tau^2k+O(\sqrt k),
\end{equation*}
each at most $\diam(U_1\cup U_2) = |U_1|+|U_2|-1 < (u_1+u_2)T\le (y+x)T$, have not been accounted for already. These contribute an additional
\begin{equation}\label{eq:U1U2 contribute together}
    (\tau_2-(y+x))\cdot 2 \alpha (1-\alpha)(1-\beta)\tau^2 k T + O(k^2)
\end{equation}
to $S(\calM, T_2)$. After some algebra to rearrange terms, this claim is confirmed.
\end{proof}

\begin{claim}\label{claim:bound from smalls}
Suppose that $y+\beta^2x \le \beta^2 \delta$. Then
    \begin{multline*}
        b_\infty \le \frac{\tau ^2 \left(2 \tau_2^2 \left(2 \alpha  \beta +\beta ^2 \delta +1\right)+2 (\alpha -1) \beta ^2 \delta  (\alpha  (2 \beta -1)-1)-2 \tau_2 (\alpha  \beta -1)^2+\tau_2^3\right)+\tau_2}{\tau  \tau_2^2}\\\le \frac{3869247756486775922024264545}{1940405707787319054606925942}\le 1.99405.
    \end{multline*}
\end{claim}

\begin{proof}
We begin with
    \begin{align*}
        \diam(\cA)&=\min\calM + \diam(\calM) + a_k-\max\calM \\
        &\ge 2(T-xT-yT) + \frac{|\calM|^2 T_2^2}{T_2(T_2+|\calM|-1)-2\tau^2 w kT}-T_2
    \end{align*}
where we have used Theorem~\ref{thm:ETexact} and Claim~\ref{claim:S}.
    The first term contributes $2\tau(1-x-y)$ to the $k^{3/2}$ coefficient, and the last term contributes $-\tau_2$. The middle term can be converted to a geometric series:
    \begin{align*}
        \frac{|\calM|^2 T_2^2}{T_2(T_2+|\calM|-1)-2\tau^2 w kT}
        &= \frac{|\calM|^2}{1+\frac{|\calM|-1}{\tau \tau_2k^{3/2}}-2\frac{\tau w}{\tau_2^2 k^{1/2}}}\\
        &= |\calM|^2\left(1+\left(\frac{2\tau w}{\tau_2^2 k^{1/2}}-\frac{|\calM|-1}{\tau\tau_2k^{3/2}}\right)+\left(\frac{2\tau w}{\tau_2^2 k^{1/2}}-\frac{|\calM|-1}{\tau\tau_2k^{3/2}}\right)^2+\cdots\right)\\
        &=|\calM|^2\left(1+\frac{2\tau w}{\tau_2^2 k^{1/2}}-\frac{|\calM|-1}{\tau\tau_2k^{3/2}}+O(1/k)\right)
    \end{align*}
Defining $\mu$ by $|\calM|=k-\mu\sqrt{k}$, we get
    \begin{align*}
        |\calM|^2\left(1+\frac{2\tau w}{\tau_2^2 k^{1/2}}-\frac{|\calM|-1}{\tau\tau_2k^{3/2}}+O(1/k)\right)
        &=(k-\mu\sqrt{k})^2\left(1+\frac{2\tau w}{\tau_2^2 k^{1/2}}-\frac{k-\mu\sqrt k}{\tau\tau_2k^{3/2}}\right) + O(k) \\
        &=k^2+ \left(\frac{2 \tau  w}{\tau_2^2} -2 \mu -\frac{1}{\tau  \tau_2} \right)k^{3/2}+O(k).
    \end{align*}

Thus,
    \begin{equation}\label{eq:binfty 2}
        b_\infty \leq \tau \tau_2-2\tau(1-x-y)+2\mu+\frac{1}{\tau\tau_2}- \frac{2 \tau  w}{\tau_2^2}.
    \end{equation}
    From Claim~\ref{claim:M}, we know that $\mu$ is at most $2(1+\alpha \beta)\tau$. The right side of~\eqref{eq:binfty 2} is a linear expression in $x$ and $y$, which we must maximize in the region $x\ge0,y\ge0, y+\beta^2 x \le \beta^2\delta$. The maximum occurs\footnote{Ideally, we would set
$$\tau_2=\sqrt{\frac{-2 \alpha ^2 \beta ^3+2 \alpha ^2 \beta ^2-\alpha ^2+2 \alpha  \beta ^3-2 \alpha  \beta +2 \alpha -\beta ^2}{\beta ^2-1}},$$
and then this expression has a factor of $y+\beta^2 x$, which we could replace with $\beta^2\delta$. The author didn't realize this until after the optimization work was complete, and including it would make some intermediate optimization steps more delicate.} at the vertex $(x,y)=(0,\beta^2\delta)$, and is
    $$b_\infty \le \frac{3869247756486775922024264545}{1940405707787319054606925942}.$$
\end{proof}

Claim~\ref{claim:bound from variance} and Claim~\ref{claim:bound from smalls} complete the proof of Theorem~\ref{thm:Main}. The value of $\delta$ was chosen so as to make the two bounds exactly equal. In general, one sets
$$\delta = -\frac{\tau_2 \left(-2 \alpha ^2 \beta ^2 \tau ^2+4 \alpha  \beta  \tau ^2 \tau_2+4 \alpha  \beta  \tau ^2+\tau ^2 \tau_2^2+\tau ^2 \tau_2-2 \tau ^2-\tau_2+1\right)}{2 \tau ^2 \left(\alpha ^2 \beta ^2 \tau_2^2+\alpha ^2 \beta ^2-\alpha ^2-2 \alpha  \beta +2 \alpha +\tau_2^2\right)}$$
to accomplish this.

In reality, the bounds were first worked out symbolically using Mathematica, and then optimized using simulated annealing. Having good estimates of where the optimal value occurs then allowed certain streamlining of the argument.

\section*{Acknowledgements}
I wish to thank Melvyn B. Nathanson and Yin Choi Cheng, who independently showed great patience while I explained early versions of this work. 

The author is supported by a PSC–CUNY Award, jointly funded by The Professional Staff Congress and The City University of New York.

\end{document}